\documentclass[11pt]{article}
\usepackage{amsmath,amssymb,amsthm,amscd,esint}
\usepackage{url,endnotes,hyperref}
\usepackage{tikz-cd}

\numberwithin{equation}{section}

\newtheorem{theorem}{Theorem}[section]
\newtheorem{lemma}[theorem]{Lemma}
\newtheorem{definition}[theorem]{Definition}

\theoremstyle{corollary}

\theoremstyle{conjecture}

\theoremstyle{assumption}

\theoremstyle{proposition}
\newtheorem{proposition}[theorem]{Proposition}
\theoremstyle{remark}
\newtheorem{remark}[theorem]{Remark}
\numberwithin{equation}{section}
\everymath{\displaystyle}

\newcommand{\Ric}{\operatorname{Ric}}
\newcommand{\Rm}{\operatorname{Rm}}
\newcommand{\Vol}{\operatorname{Vol}}
\newcommand{\diam}{\operatorname{diam}}

\newcommand{\Bl}{\operatorname{Bl}}
\newcommand{\Spec}{\operatorname{Spec}}

\newcommand{\PP}{\mathbb P}

\newcommand{\OO}{\mathcal O}

\begin{document}

\title{Finite Time Type I Singularities of the K\"ahler Ricci Flow}
\author{}
\date{}
\maketitle

\vspace{-3em}
\begin{center}
\Large
\begin{tabular}{cc}
Tongxin Xu$^*$ & \qquad Zhenlei Zhang$^\dagger$
\end{tabular}
\end{center}
\vspace{0.5em}

\begin{abstract}
We prove the Feldman–Ilmanen–Knopf conjecture for finite time
Type I singularities of the Kähler–Ricci flow. More precisely, for any compact Kähler manifold $Y$ and its blow-up $\pi:\Bl_pY\longrightarrow Y$, if $[\omega_0]-Tc_1(M)=\pi^*[\omega_Y]$, then any Type I parabolic blow-up limit of the K\"ahler Ricci flow along the exceptional divisor is the FIK shrinker $\operatorname{Tot}(\OO_{\PP^{n-1}}(-1))$.
\end{abstract}
\maketitle
\tableofcontents

\section{Introduction}

Ricci flow was introduced by Hamilton  as an evolution equation for Riemannian metrics in 1980s \cite{hamilton1982three}. In this paper, we mainly consider the Kähler Ricci flow. Let $(M,\omega_{0})$ be a compact Kähler manifold of complex
dimension $n$, and let $\omega(t)$ be the solution of the unnormalized
Kähler--Ricci flow
\begin{equation*}
    \frac{\partial}{\partial t}\omega(t)
    =
    -\Ric(\omega(t)),
    \qquad
    \omega(0)=\omega_{0},
\end{equation*}
on its maximal time interval $[0,T)$. Throughout this paper, we assume
that the maximal existence time is finite, namely $T<\infty$. The flow is
said to develop a Type~I singularity at time $T$ if
\begin{equation}\label{eq:type-I}
    \sup_{M}|\Rm(\omega(t))|_{\omega(t)}
    \leq \frac{C}{T-t}
\end{equation}
for some uniform constant $C<\infty$ and every $t\in[0,T)$.

A fundamental problem in the study of finite-time singularities is to
understand the geometry arising under parabolic rescaling. By the work
of Naber and Enders--Müller--Topping, suitable blow-ups based at a
Type~I singular point subconverge to a complete nonflat Ricci shrinker with bounded curvature \cite{Naber,EMT}. In the
Kähler setting, the limit is naturally a Kähler Ricci shrinker.

Finite-time singularities of the Kähler--Ricci flow are closely related
to birational geometry. Song--Weinkove established Gromov--Hausdorff
convergence and smooth convergence away from the exceptional locus for
flows contracting exceptional divisors~\cite{SongWeinkove2013}. The corresponding local Type~I model is expected
to be the $U(n)$-invariant FIK shrinker on $\operatorname{Tot}
    \bigl(\mathcal{O}_{\mathbb{P}^{n-1}}(-1)\bigr)$ \cite{FeldmanIlmanenKnopf2003}. This picture was confirmed in several
symmetric settings. Song proved that finite-time non-collapsed $U(n)$-invariant
Kähler--Ricci flows on $\Bl_{p}\mathbb{P}^{n}$ develop Type~I
singularities \cite{Song2015}, and Guo--Song identified the
volume-noncollapsing blow-up limit with the unique $U(n)$-invariant FIK
shrinker \cite{GuoSong2016}. In complex dimension two, related results
were obtained by Máximo \cite{max14}.

Considerably more is known in Kähler surfaces. Cifarelli--Conlon--Deruelle proved that a noncompact blow-up
limit of a finite-time Type~I singularity on a compact Kähler surface is
the FIK shrinker if and only if the flow is volume-noncollapsing
\cite{CCD}. More recently, Conlon--Hallgren--Ma showed that every
finite-time volume-noncollapsing singularity on a compact Kähler surface
is automatically of Type~I \cite{CHM}. Consequently, every such
singularity is modeled by the FIK shrinker. This gives a rigid
description of finite-time noncollapsing singularities in complex
dimension two without assuming the Type~I condition in advance.
In contrast, for complex dimension at least three, there are finite-time singularities of the flow which are not of Type~I \cite{LTZ24,MT23}.

Bamler developed a general compactness and structure theory for
noncollapsed Ricci flows in the framework of metric flows
\cite{bamler2020entropy,BamlerCompactness,bamler2020structure}. Building on
this framework, Jian--Song--Tian proved that singular-time-scale
rescalings of general finite-time Kähler--Ricci flows, based at
suitable points, subconverge to ancient solutions on families of
analytic normal varieties \cite{JianSongTian2023}.
Hallgren--Jian--Song--Tian further established continuity in time and
codimension at least four of the singular sets of these limits
\cite{HallgrenJianSongTian2024}.

For any compact Kähler manifold $(Y,\omega_Y)$, we consider the finite-time Type~I Kähler--Ricci flow
$(M,\omega(t))=(\operatorname{Bl}_{p}Y,\omega(t))$, $t\in[0,T)$, satisfying
$[\omega_0]-Tc_1(M)=\pi^*[\omega_Y]$ and we set the exceptional divisor to be $E$.
We prove the following theorem:
\begin{theorem}\label{main}
Let $(M,g(t))$ be a finite time Type I Kähler Ricci flow on $[0,T)$, if $[\omega_0]-Tc_1(M)=\pi^*[\omega_Y]$, then for every sequence $\lambda_j \to \infty$ and any $x\in E$, the rescaled Kähler Ricc flows $(M,g_j(s),x)$ defined by $g_j(s):=\lambda_jg(T+{s}/{\lambda_j})$ on $[-\lambda_jT,0)$ subconvergence in pointed $C^{\infty}$-Cheeger-Gromov topology to the unique FIK shrinker on $\operatorname{Tot}(\OO_{\PP^{n-1}}(-1))$.
\end{theorem}

\paragraph{Outline of the proof.} 
The proof consists of an analytic part and an algebraic part.

Let $(M,g(t))$ be a finite time Type I Kähler Ricci flow on $[0,T)$. Fix \(p\in E\). Type~I
compactness give a smooth pointed limit $(X,g_\infty,J_\infty,p_\infty)$, which is a complete nonflat Kähler Ricci shrinker. A uniform diameter bound for the rescaled exceptional divisors allows
us to pull the entire divisor \(E\) back by the Cheeger--Gromov maps.
The resulting hypersurfaces \(E_i\subset X\) remain in a fixed compact
set and have constant volume \(V_0\).  Bishop compactness therefore
gives an effective compact analytic cycle
\[
        [E_i]\rightharpoonup\mathcal Z,
        \qquad
        Z:=\operatorname{Supp}\mathcal Z,
        \qquad
        \mathbf M_{g_{\infty}}(\mathcal Z)=V_0.
\]
The logarithmic Ricci potential and localization argument of \cite[Proposition 3.2, 5.1]{CHM} shows that every positive-dimensional compact
analytic subvariety of \(X\) is contained in \(Z\).  Thus \(Z\) is the
maximal compact analytic subset of \(X\).

By Sun-Zhang \cite[Theorem 3.1]{SZ}, the K\"ahler--Ricci shrinker $X$ admits a
polarized Fano fibration $q:X\longrightarrow W$, where \(W\) is a normal positively graded affine variety and
\(-K_X\) is \(q\)-ample.  The maximality of \(Z\) implies that \(q\)
is birational and
\[
        q^{-1}(o)=Z,
        \qquad
        X\setminus Z\simeq W\setminus\{o\},
\]
where \(o\) is the vertex of \(W\).

Then we study this birational morphism $q$. For a relative extremal ray
\(R\subset\overline{\operatorname{NE}}(X/W)\), the length of relative extremal ray exhibits rigidity. A local Chern-class
calculation gives \(l(R)\geq n-1\), while the relative
Ionescu--Wi\'sniewski inequality gives \(l(R)\leq n-1\).  Hence $l(R)=n-1$.
The local adjoint-contraction theorem then produces an exceptional
divisor
\[
        D_R\simeq\PP^{n-1},
        \qquad
        N_{D_R/X}\simeq\OO_{\PP^{n-1}}(-1),
\]
contracted to a smooth point.  Since \(D_R\subset Z\) and
\(\operatorname{Vol}(D_R)=V_0=\operatorname{Mass}(\mathcal Z)\), the
limiting cycle is exhausted by \(D_R\):
\[
        \mathcal Z=[D_R],
        \qquad
        Z=D_R.
\]
It follows that there is only one relative extremal ray and that \(q\)
is the ordinary blow-up of \(W\) at \(o\). Finally, the positive grading and the smoothness of \(W\) at its vertex
imply that $W\simeq\mathbb C^n$. Consequently,
\[
        X\simeq\operatorname{Bl}_0\mathbb C^n
          \simeq\operatorname{Tot}\OO_{\PP^{n-1}}(-1).
\]
By the uniqueness theorem of \cite[Theorem E]{CDS}, we know $X$ is the FIK shrinker.

\section{Preliminaries}
\subsection{Kähler Ricci flow }
\subsubsection{Finite time Type I Kähler Ricci flow}
We first fix some setup of cohomology class. Let $\pi:M=\operatorname{Bl}_pY\to Y$ be the blow-up of a compact Kähler manifold at \(p\), with exceptional divisor \(E\simeq\mathbb P^{n-1}\). Writing $[\omega_0]=\pi^*\alpha_0-a_0[E]$ with $\alpha_0\in H^{1,1}(Y;\mathbb{R})$ and $a_0>0$, the cohomological evolution of the unnormalized Kähler–Ricci flow is
\[
[\omega(t)]
=\pi^*(\alpha_0-tc_1(Y))
-\bigl(a_0-(n-1)t\bigr)[E].
\]
Under the pure point-contraction assumption
\begin{equation}\label{2}
     [\omega_0]-Tc_1(M)=\pi^*[\omega_Y],\qquad \omega_Y>0,   
\end{equation}
we have \(T=a_0/(n-1)\). We set $h=c_1\bigl(\mathcal O_{\mathbb P^{n-1}}(1)\bigr)$  and hence
\begin{equation}\label{3}
      [\omega(t)]|_E=(n-1)(T-t)h,
\end{equation}
Moreover, $\lim_{t\nearrow T}\operatorname{Vol}_{\omega(t)}(M)
=\frac1{n!}\int_Y\omega_Y^n>0$, so the flow is automatically non-collapsing.
\begin{definition}\label{soliton}
    A K\"ahler Ricci shrinker is a quadruple $(X,g,J,f)$, where $X$ is a complete K\"ahler manifold with a real holomorphic vector field $\nabla^gf$ satisfying the equation
\[
\operatorname{Ric}(\omega) + \sqrt{-1}\partial \bar{\partial} f = \omega,
\]
where $\operatorname{Ric}(\omega) $ is the Ricci form of Kähler form $\omega$.
\end{definition}
Kähler Ricci shrinker is a self-similar solution of Kähler Ricci flow. More precisely, we can define $g(t):=|t|\varphi_t^*g$, where $\varphi_t$ is a family of biholomorphisms generated by vector field $\frac{\nabla^g f}{|t|}$ with $\varphi_{-1}=\operatorname{id}$. The one-parameter family of metrics $(X,g(t),J)_{t\in(-\infty,0)}$ defines a Kähler Ricci flow and it satisfies the Type I curvature condition naturally.
\begin{definition}\cite[Definition 1.2]{EMT}
Let $(M,g(t))$, $t\in[0,T)$, $T<+\infty$, be a K\"ahler--Ricci flow.
A space-time sequence $(x_i,t_i)$ with $x_i\in M$ and $t_i\to T^-$ is called an \textit{essential blowup sequence} if there exists a constant $C>0$ such that
\[
|\operatorname{Rm}_{g(t_i)}|_{g(t_i)}(x_i)\ge \frac{C}{T-t_i}.
\]
A point $x\in M$ in a Type I K\"ahler--Ricci flow is called a \textit{Type I singular point} if there exists an essential blowup sequence with $x_i\to x$ on $M$.
We denote the set of all Type I singular points by $\Sigma_I$.
\end{definition}

For finite time K\"ahler--Ricci flow, the norm of the curvature $|\mathrm{Rm}|_{g(t)}$ tends to infinity as $t\to T^-$. 
Moreover, in Type I condition, the parabolic maximum principle implies that the curvature blowup is equivalent to
\[
\sup_{M} \left|\operatorname{Rm}_{g(t)}\right|_{g(t)} \ge \frac{1}{8(T-t)} \qquad \text{for all } t\in[0,T),
\]
so the above definition is introduced precisely to single out the non-trivial singular points that yield non-flat rescaled limits.

By \cite{EMT,Naber}, any essential blowup sequence of Type I singularities converges to a non-trivial K\"ahler Ricci shrinker. 
More precisely, for every sequence $\lambda_j \to \infty$ and any $x\in \Sigma_I$, the pointed rescaled Ricc flows $(M,g_j(s),x)$ defined by $g_j(s):=\lambda_j g(T+s/\lambda_j)$ on$[-\lambda_jT,0)$ subconverge in the pointed $C^\infty$-Cheeger--Gromov sense to a complete K\"ahler Ricci shrinker with bounded curvature:
\[
\bigl(M, g_j(s), J_M, x\bigr)_{s\in[-\lambda_jT,0)} 
\xrightarrow{\text{pointed-}C^\infty\text{-Cheeger--Gromov}} 
\bigl(X, g_\infty(s), J_\infty, x_\infty\bigr)_{s\in(-\infty,0)}.
\]
Furthermore, the smooth convergence together with the divergence of the rescaled volume forces the limit $X$ to be noncompact.

In our setting, $(M,\omega(t))=(\operatorname{Bl}_{p}Y,\omega(t))$. 
By \cite[Theorem 1.5, 4.7]{CT} and \cite[Theorem 1.8]{EMT}, we have $E=\operatorname{Null}(\pi^*\omega_Y)=\Sigma_{I}$. 
Thus, the Type I singular points correspond precisely to the points lying on the exceptional divisor.

\subsubsection{Parabolic Schwarz lemma and localization}
For the finite time singularity Kähler Ricci flow, the pure point contraction condition implies the following well-known lemma.
\begin{lemma}\cite[Lemma 3.7.3]{swlecture}\label{Schwarz}
    Suppose there exists a holomorphic map $f : M \to N$ to a compact Kähler manifold $N$ and let $\omega_N$ be a Kähler metric on $N$. We assume that $[\omega_0] - Tc_1(M) = [f^*\omega_N]$. Then on $M \times [0, T)$, $\omega \ge {C}^{-1} f^*\omega_N$ for a uniform constant $C$.
\end{lemma}
The above parabolic Schwarz lemma can induce the following localization lemma.

\begin{lemma}\label{lem:localization}
    Let \(K\Subset X\) and let \(\Phi_i:K\to M\) be Cheeger--Gromov
embeddings at time \(-1\). Then
\[
        \pi(\Phi_i(K))
        \subset B_{\omega_Y}(p,C_K \lambda_i^{-1/2})
\]
for all sufficiently large \(i\). The analogous assertion holds uniformly
on compact spacetime subsets of the limiting ancient flow.
\end{lemma}
\begin{proof}
After enlarging \(K\), assume every point of \(K\) has
\(\omega_\infty(-1)\)-distance at most \(R_K\) from \(x_\infty\).
Cheeger--Gromov convergence gives $d_{\omega(t_i)}(\Phi_i(x),x)
        \leq 2R_K\lambda_i^{-1/2}$. The Lemma \ref{Schwarz} implies that the projection of every curve has
\(\omega_Y\)-length at most a fixed multiple of its \(\omega(t_i)\)-length.
Since \(\pi(x)=p\), the conclusion follows. The spacetime version is
uniform on compact rescaled time intervals.
\end{proof}

\begin{lemma}
\label{loc}
Suppose that for every compact spacetime set $\mathcal K\Subset X\times(-\infty,0)$, there is a constant \(C_{\mathcal K}<\infty\), independent of \(i\),
such that $\pi\bigl(\Psi_i(\mathcal K)\bigr)
    \subset
    B_{\omega_Y}\bigl(p,C_{\mathcal K}\lambda_i^{-1/2}\bigr)$. Then, for every open neighborhood \(V\subset M\) of \(E\), there
exists \(i_0=i_0(\mathcal K,V)\) such that $\Psi_i(\mathcal K)\subset V\times(-\infty,0)$ for every \(i\ge i_0\).
\end{lemma}

\begin{proof}
Since \(M\) is compact, the set \(M\setminus V\) is compact. Moreover, $p\notin\pi(M\setminus V)$, because $\pi^{-1}(p)=E\subset V$. Consequently, $d_{\omega_Y}\bigl(p,\pi(M\setminus V)\bigr)>0$. Thus there exists \(r_V>0\) such that $\pi^{-1}\bigl(B_{\omega_Y}(p,r_V)\bigr)\subset V$. By $ \pi\bigl(\Psi_i(\mathcal K)\bigr)
    \subset
    B_{\omega_Y}\bigl(p,C_{\mathcal K}\lambda_i^{-1/2}\bigr)$. Since \(C_{\mathcal K}\) is independent of \(i\) and
\(\lambda_i\to \infty\) for sufficiently large \(i\) one has
$C_{\mathcal K}\lambda_i^{-1/2}<r_V$. The conclusion now follows.
\end{proof}

\subsection{Polarized Fano fibration}
Here we recall the polarized Fano fibration
framework of \cite{SZ}.

A \textit{fibration} is a surjective projective morphism $\pi \colon X \to Y $
between normal varieties such that $\pi_*\mathcal{O}_X = \mathcal{O}_Y$. The following definition is introduced by \cite{collins2019sasaki} to study the canonical metric in Sasaki geometry.
\begin{definition}\cite{collins2019sasaki}
    A \textit{polarized affine cone} $(Y,\xi,T)$ consists of a normal affine variety $Y = \operatorname{Spec} R$ with a compact torus $T$-action admitting a unique fixed point, together with a vector $\xi \in \operatorname{Lie}(T)$ lying in the Reeb cone. Here the latter condition means that for the weight decomposition
\[
R = \bigoplus_{\alpha \in \operatorname{Lie}(T)^*} R_\alpha,
\]
one has $\langle \alpha, \xi \rangle > 0$ whenever $R_\alpha \neq 0$ and $\alpha \neq 0$.
\end{definition}
\begin{definition}[Polarized Fano fibration]
    A polarized Fano fibration $(\pi\colon X \to Y, \xi)$ is a fibration $\pi\colon X \to Y$ with the following properties:
\begin{enumerate}
    \item $\pi\colon X \to Y$ is a Fano fibration. That is, $X$ and $Y$ are normal varieties, $X$ is klt, and that $-K_X$ is $\pi$-ample and $\mathbb{Q}$-Cartier.
    \item $X$ and $Y$ are equipped with a $\pi$-equivariant torus action $T$, and $\xi \in \mathfrak{t} = \operatorname{Lie}(T)$.
    \item $(Y, T, \xi)$ is a polarized affine cone.
\end{enumerate}
\end{definition}
Sun-Zhang proved the following structural theorem via deep algebraic geometric techniques.
\begin{theorem}\cite[Theorem 3.1]{SZ}
    A K\"ahler-Ricci shrinker $(X,g,J,f)$ naturally defines a polarized Fano fibration. In particular, $X$ is quasi-projective.
\end{theorem}
The following lemma is well‑known in commutative algebra and will be used to clarify the base $Y$.
\begin{lemma}\label{lem:smooth-cone}
If \((Y,\xi,T)\) is smooth at \(o\), then $Y\simeq\mathbb C^n$ as an affine variety.
\end{lemma}

\begin{proof}
Write \(Y=\operatorname{Spec}A\). Since the Reeb cone is open and
rational points are dense in \(\operatorname{Lie}(T)\), a sufficiently
small rational perturbation of \(\xi\), followed by rescaling, defines
an algebraic one-parameter subgroup of \(T_{\mathbb C}\). We therefore
obtain a positive integral grading
\[
        A=\bigoplus_{k\geq0}A_k,
        \qquad
        A_0=\mathbb C,
        \qquad
        \mathfrak m_o=\bigoplus_{k>0}A_k .
\]

Since \(o\) is smooth and \(\dim Y=n\), we have $\dim_{\mathbb C}\mathfrak m_o/\mathfrak m_o^2=n$. Choose homogeneous elements \(z_1,\ldots,z_n\in\mathfrak m_o\) whose
classes form a basis of \(\mathfrak m_o/\mathfrak m_o^2\). By the
graded Nakayama lemma \cite[Tag~0EKB]{stacks-project}, together with
induction on the degree, these elements generate \(A\) as a
\(\mathbb C\)-algebra. Hence we have a surjection $\mathbb C[x_1,\dots,x_n]\rightarrow A$, $x_j\mapsto z_j$. Its kernel is prime because \(A\) is a domain. Since both rings have
Krull dimension \(n\), the kernel has height zero and is therefore
zero. Thus $A\simeq\mathbb C[x_1,\ldots,x_n]$, and consequently \(Y\simeq\mathbb C^n\).

\end{proof}

\section{Analytic part}
\subsection{Extrinsic diameter bound and compact analytic set}
Firstly, we establish the following extrinsic diameter bound.
\begin{lemma}\label{diam}
There is a constant \(D<\infty\), independent of \(i\), such that
\[
        \diam_{\widehat g_i}^{\mathrm{ext}}(E)\leq D.
\]
The same assertion holds uniformly on every compact rescaled time interval
\(s\in[-B,-B^{-1}]\Subset(-\infty,0)\).
\end{lemma}

\begin{proof}
We set $\widehat\omega_i:=\omega_i(-1)=\lambda_i\omega(T-1/\lambda_i)$. By (\ref{3}), we have
\begin{equation}\label{3.1}
 \Vol_{\widehat\omega_i}(E)
 =\frac1{(n-1)!}\int_E\widehat\omega_i^{\,n-1}
 =V_0:=
 \frac{(n-1)^{n-1}}{(n-1)!}
 \int_{\PP^{n-1}}h^{n-1}.
\end{equation}

Since \(E\) is a complex hypersurface in the K\"ahler manifold
\((M,\widehat\omega_i)\), it is \(\widehat\omega_i\)-minimal. The Type I condition gives
uniform ambient bounded geometry.
The local monotonicity formula applied in
uniform normal coordinates
\cite[Chapter~8, \S3]{simon2014introduction},
therefore gives constants \(r_0,c>0\), independent of \(i\), such that
\begin{equation}\label{3.2}
     \mathcal H_{\widehat\omega_i}^{2n-2}
    \bigl(E\cap B_{\widehat\omega_i}(x,r)\bigr)
    \geq c\,r^{2n-2}   
\end{equation}
for every \(x\in E\) and \(0<r\leq r_0\).
Choose a maximal $2\delta$-separated collection $\{x_1,\dots,x_N\} \subset E$. The balls $B(x_j,\delta)$ are pairwise disjoint, so (\ref{3.1}) and (\ref{3.2}) imply $Nv \leq V_0$. Maximality gives that $E$ is contained in the union of the finite collection of balls $B(x_j,2\delta)$ for $j=1,...,N$.
The intersection graph of the relatively open sets $E \cap B(x_j,2\delta)$ is connected because $E$ is connected. Therefore
\[
\operatorname{diam}_{\widehat{g}_i}^{\operatorname{ext}}(E) \leq 4N\delta \leq 4\delta V_0/v.
\]
The argument on compact rescaled time intervals is identical.
\end{proof}

Now we construct the compact analytic set in limit space $X$, the main discussion comes from \cite[Lemma 3.1]{CHM}.

\begin{proposition}\label{analy}
    $Z$ is a compact connected analytic set with pure dimension $n-1$ in $X$.
\end{proposition}
\begin{proof}
By Lemma~\ref{diam}, there exists \(D<\infty\) such that $\operatorname{diam}^{\mathrm{ext}}_{\widehat g_i}(E)
        \leq D$. Since the Cheeger--Gromov basepoint \(x\) lies in \(E\), $E\subset B_{\widehat g_i}(x,D)$. Choose the Cheeger--Gromov maps $\Phi_i:U_i\longrightarrow V_i\subset M$ with $\Phi_i(x_\infty)=x$ so that $B_{\widehat g_i}(x,D+1)\subset V_i$ for all sufficiently large \(i\). Then $E_i:=\Phi_i^{-1}(E)$ is well-defined and the sets \(E_i\) are contained in a fixed compact
subset \(K\Subset X\). After passing to a subsequence, they converge
in the Hausdorff topology of \(K\) to a nonempty compact set
\(Z\subset K\). Since each \(E_i\) is connected, so is
its Hausdorff limit \(Z\).

Set $h_i:=\Phi_i^*\widehat g_i$ and $J_i:=\Phi_i^*J_M$. Then $(h_i,J_i)\longrightarrow(g_\infty,J_\infty)$ in $C^\infty_{\mathrm{loc}}(X)$, and each \(E_i\) is a \(J_i\)-complex hypersurface. Moreover,
\begin{equation}\label{mmm}
 \mathcal H_{h_i}^{2n-2}(E_i)
 =
 \frac{1}{(n-1)!}
 \int_{E_i}(\Phi_i^*\widehat\omega_i)^{n-1}
 =
 \frac{1}{(n-1)!}
 \int_E\widehat\omega_i^{\,n-1}
 =
 V_0.
\end{equation}

Fix \(y\in Z\). By the convergent holomorphic-coordinate lemma
\cite[Lemma~A.1]{CHM}, after shrinking a neighborhood \(U\) of $y$,
there exist \(J_i\)-holomorphic coordinates \(z_i\) converging
smoothly to a \(J_\infty\)-holomorphic coordinate system
\(z_\infty\). On every relatively compact coordinate subdomain, the
sets \(z_i(E_i\cap U)\) are pure \((n-1)\)-dimensional analytic sets.
The smooth convergence of the coordinates and metrics, together with
\eqref{mmm}, gives the locally uniform volume bound. Since \(E_i\to Z\) locally in the
Hausdorff topology, Bishop's theorem
\cite[Theorem~1.7]{DiederichMazzilli} shows that $z_\infty(Z\cap U)$ is a pure \((n-1)\)-dimensional analytic subset. So \(Z\) is a compact \(J_\infty\)-analytic subset of \(X\)
of pure dimension \(n-1\).
\end{proof}
In fact, we can obatin the following proposition of the convergence as analytic cycles.
\begin{proposition}
\label{cycle}
After passing to a subsequence, there exist distinct irreducible
compact \(J_\infty\)-analytic subvarieties $Z_1,\ldots,Z_N\subset X$ with $\dim_{\mathbb C}Z_\alpha=n-1$, and integers \(m_\alpha\geq1\) such that
\[
        [E_i]\rightharpoonup
        \mathcal Z:=
        \sum_{\alpha=1}^{N}m_\alpha[Z_\alpha]
\]
weakly as currents on \(X\). Moreover,
\[
        Z=\operatorname{Supp}\mathcal Z
          =\bigcup_{\alpha=1}^{N}Z_\alpha,
          \qquad
        \mathbf M_{g_\infty}(\mathcal Z)
        =
        \frac{1}{(n-1)!}
        \left\langle
        \mathcal Z,\omega_\infty^{n-1}
        \right\rangle
        =V_0.  
\]
\end{proposition}
\begin{proof}
     Let $T_i:=[E_i]$ be the integral current of integration over \(E_i\). Since \(E_i\) is compact and has no
boundary, $\partial T_i=0$. Moreover, \(\operatorname{supp}T_i\subset K\), and by
\eqref{mmm}, we know $\mathbf M_{h_i}(T_i)=V_0$. Since \(h_i\to g_\infty\) smoothly on the fixed compact set \(K\), the
metrics \(h_i\) and \(g_\infty\) are uniformly equivalent there.
Consequently, $\sup\nolimits_{i}\mathbf M_{g_\infty}(T_i)<\infty$. The Federer--Fleming compactness theorem for integral currents
\cite[4.2.17]{Federer} therefore gives, after passing to a further
subsequence, an integral current \(T\) such that $T_i\rightharpoonup T$ and $\partial T=0$. Because \(T_i\) is \(J_i\)-complex and
\(J_i\to J_\infty\) smoothly on \(K\), the limiting current \(T\) is
a positive integral current of bidimension \((n-1,n-1)\) with respect
to \(J_\infty\). King's characterization of positive holomorphic
chains \cite{King} or \cite[Corollary 3.10.2]{TehYang}  then implies that
\[
        T=\sum_{\alpha=1}^{N}m_\alpha[Z_\alpha],
        \qquad
        m_\alpha\in\mathbb N_{>0},
\]
where the \(Z_\alpha\) are distinct irreducible
\((n-1)\)-dimensional \(J_\infty\)-analytic subvarieties of \(X\).
Since \(\operatorname{supp}T\subset K\), only finitely many
irreducible components occur. We denote this limiting analytic cycle
by $\mathcal Z:=T$.

We next identify the support of the limiting cycle. The Hausdorff
convergence \(E_i\to Z\) immediately gives $\operatorname{Supp}\mathcal Z\subset Z$. Indeed, every compactly supported test form in \(X\setminus Z\)
vanishes on \(E_i\) for all sufficiently large \(i\). 

Conversely, fix \(y\in Z\), and choose \(y_i\in E_i\) with
\(y_i\to y\). Set \(p=n-1\) and $\eta_i:=\Phi_i^*\widehat\omega_i$. For every fixed sufficiently small \(r>0\), the Lelong monotonicity
formula \cite[Chapter~III, \S5, Consequence~(5.8)]{Dema} in the convergent holomorphic coordinates, together with the
uniform equivalence of the coordinate metrics, gives 
\begin{align*}
        \mathbf M_{h_i} \bigl(T_i\llcorner B_{g_\infty}(y,r)\bigr) \geq c r^{2p}
\end{align*}
for sufficiently large \(i\), where \(c>0\) is independent of
\(i\). Choose \(\chi\in C_c^\infty(B_{g_\infty}(y,2r))\) with
\(0\leq\chi\leq1\) and \(\chi\equiv1\) on
\(B_{g_\infty}(y,r)\). Since \(T_i=[E_i]\) is calibrated by
\(\eta_i^p/p!\), we have
\begin{align*}
        T_i\left(\chi({\eta_i^p}/{p!})\right) \geq c r^{2p}.
\end{align*}
 Using \(T_i\rightharpoonup T\), the smooth convergence
\(\eta_i\to\omega_\infty\), and the uniform mass bound for \(T_i\),
we may pass to the limit and obtain 
\begin{align*}
        T\left(\chi({\omega_\infty^p}/{p!})\right) \geq c r^{2p}>0.
\end{align*}
Thus \(T\) is nonzero in every neighborhood of \(y\), so
\(y\in\operatorname{Supp}\mathcal Z\). Therefore $\operatorname{Supp}\mathcal Z=Z$.

Finally, the smooth Cheeger--Gromov convergence gives $\Phi_i^*\widehat\omega_i\longrightarrow\omega_\infty$
in $C^\infty(K)$. Choose \(\chi_0\in C_c^\infty(X)\) such that $0\leq\chi_0\leq1$ and $\chi_0\equiv1$  on a neighborhood of the fixed compact set \(K\) containing
\(E_i\) and \(Z\). Since \(\mathcal Z\) is supported in \(K\), the weak convergence \(T_i\rightharpoonup\mathcal Z\) then gives
\begin{align*}
 \frac{1}{(n-1)!}
 \left\langle\mathcal Z,\omega_\infty^{n-1}\right\rangle
 =
 \lim_{i\to\infty}
 \frac{1}{(n-1)!}
 \left\langle T_i,
        \chi_0\omega_\infty^{n-1}\right\rangle
 =
 \lim_{i\to\infty}
 \frac{1}{(n-1)!}
 \left\langle T_i,
        \bigl(\Phi_i^*\widehat\omega_i\bigr)^{n-1}\right\rangle
 =V_0.
\end{align*}
Therefore $\mathbf M_{g_\infty}(\mathcal Z)=V_0$.
\end{proof}

\subsection{Logarithmic Ricci function and maximal compact analytic set}
We first review the logarithmic Ricci function introduced by \cite[Section 3]{CHM}. Choose a smooth positive volume form \(\Omega_Y\) on \(Y\) such that $\theta:=\Ric(\Omega_Y)$ vanishes near \(p\). Define
\[
        u_t=\log\frac{\pi^*\Omega_Y}{\omega_t^n},
\]
which is smooth on $(M\setminus E)\times[0,T]$, but has logarithmic singularity along $E$. In ordinary blow-up coordinates, $\pi(z_1,\ldots,z_n)=(z_1,z_1z_2,\ldots,z_1z_n)$, so the complex Jacobian vanishes to order \(n-1\) along \(E\).

\begin{proposition}\cite[Proposition 3.2]{CHM}\label{heat}
    \begin{enumerate}
\item  There is a defining section $s\in H^0 (M,\mathcal{O}_{M}(E)$ and a smooth time-dependent Hermitian metric $h_t$ on $\mathcal{O}_{M}(E)$ such that ${u}_t =(n-1) \log |s|_{h_t}^2$.
\item In the sense of currents, $\sqrt{-1} \partial \bar{\partial} u_t =\operatorname{Ric}_{{\omega}_t} - \pi^{\ast} \theta +2\pi (n-1)[E]$, where $[E]$ is the current of integration along $E$.
\item $ (\partial_t-\Delta_{\omega_t})u_t
=
 \operatorname{tr}_{\omega_t}\pi^*\theta
 -2\pi(n-1)\mathcal H^{2n-2}_{g_t}|_E $, where $\mathcal{H}_{{g}_t}^{2n-2}$ denotes the $(2n-2)$-dimensional Hausdorff measure with respect to the metric $d_{g_t}$.
\item For any $\epsilon>0$, there exists $C(\epsilon)>0$ such that any $(x,t) \in ({M}\setminus {E})\times [\frac{T}{2},T)$ with 
$$P^{\ast-}(x,t;\epsilon \sqrt{T-t}) \cap ({E} \times [0,T)) =\emptyset$$ satisfies 
$|\nabla {u}_t|(x) \leq \frac{C(\epsilon)}{\sqrt{T-t}}$.
\end{enumerate}
\end{proposition}
\begin{proof}
Parts 1-3 follow from the local expression of the blow-down
map, the Poincaré--Lelong formula, and taking the trace of the
resulting current identity.

For part 4, we follow the proof of
\cite[Proposition~3.2(iv)]{CHM}. The only dimensional changes are
that the heat-kernel gradient estimate in real dimension \(2n\)
contains the factor \((t-s)^{-n-\frac12}\). Consequently, the singular term is bounded by
\[
 C\int_0^t
 \frac{(T-s)^{n-1}}{(t-s)^{n+\frac12}}
 \exp\left(
   -\frac{d_s(z_s,E)^2}{10(t-s)}
 \right)ds.
\]
Splitting the integral at
\(t-c\epsilon^2(T-t)\) and using the \(P^{\ast-}\)-separation
exactly as in \cite[Proposition~3.2(iv)]{CHM} yields $|\nabla u_t|_{g_t}(x)
       \le C_\epsilon(T-t)^{-1/2}$.

\end{proof}

Then we show that $E$ is the maximal compact analytic subset of $X$. This proof follows the proof of \cite[Proposition~5.1]{CHM}.
\begin{proposition}\label{maximalcompact}
    Every positive-dimensional compact analytic subvariety of
\(X\) is contained in \(Z\).
\end{proposition}
\begin{proof}
\textbf{Claim:} Let \(K\Subset X\setminus Z\) and \(B>1\). Then there exist
\(r=r(K,B)\in(0,1)\) and \(i_0=i_0(K,B)\) such that $P^{\ast-}_{g_i}
 \bigl(\Psi_{i,s}(x),s;r\bigr)
 \cap\bigl(E\times[-B-1,-1]\bigr)=\varnothing$ for every \(i\ge i_0\), \(x\in K\), and \(s\in[-B,-1]\).

\begin{proof}
    For each \(s\in[-B-1,-1]\), set $E_{i,s}:=\Psi_{i,s}^{-1}(E)$.
Since \(K\cap Z=\varnothing\), compactness, the spacetime Hausdorff
convergence $E_{i,s}\longrightarrow Z$ locally uniformly for \(s\in[-B-1,-1]\), and the smooth Cheeger--Gromov convergence give $d_{g_i(s)}\bigl(\Psi_{i,s}(x),E\bigr)\ge d_0>0$ for any $x\in K$ and $s\in[-B-1,-1]$ for all sufficiently large \(i\). The Type~I curvature condition and
\cite[Corollary 9.6]{bamler2020entropy} yields a constant
\(A<\infty\) such that
\[
 P^{\ast-}_{g_i}
 \bigl(\Psi_{i,s}(x),s;r\bigr)
 \subset
 \bigcup_{\sigma\in[s-r^2,s]}
 B_{g_i(\sigma)}
 \bigl(\Psi_{i,\sigma}(x),Ar\bigr)
 \times\{\sigma\}
\]
whenever \(r\) is sufficiently small. Choosing
\(Ar<d_0/2\) proves the assertion.
\end{proof}
Fix \(x_\ast\in X\setminus Z\), and define $\widetilde u_i(y,s):= u_{T+\lambda_i^{-1} s}(y)+a_i$, where $a_i:=-u_{t_i}\bigl(\Psi_{i,-1}(x_\ast)\bigr)$. Thus \(\widetilde u_i\) is defined on the rescaled flow on \(M\).
For every \(s<0\), define the pulled-back quantities on \(X\) by $\bar g_i(s):=\Psi_{i,s}^*g_i(s)$ and $J_i(s):=\Psi_{i,s}^*J_M$. And we set $v_i(\cdot,s):=\Psi_{i,s}^*\widetilde u_i(\cdot,s)$, where $\Psi_{i,s}$ is the spacetime Cheeger-Gromov mapping. Let \(L\Subset X\setminus Z\) be compact and connected with
\(x_\ast\in L\), and let \(B>1\). The \(P^*\)-separation claim,
together with the gradient estimate for Proposition \ref{heat}, gives
\begin{equation*}
    \sup_{s\in[-B,-1]}
    \sup_{x\in L}
    \left(
        |v_i(x,s)|
        +
        |\nabla^{\bar g_i(s)}v_i(x,s)|_{\bar g_i(s)}
    \right)
    \leq C(L,B).
\end{equation*}
Here the \(C^0\)-estimate follows from the normalization, the gradient
estimate, the Type~I scalar-curvature bound, and the smooth control of
the spacetime convergence maps.

By Lemma \ref{loc}, the image of
\(L\times[-B,-1]\) under the convergence maps is eventually contained
in the fixed exceptional neighborhood \(V\), on which the background
form \(\pi^*\theta\) vanishes. Since the corresponding parabolic
neighborhoods are disjoint from \(E\), the functions
\(\widetilde u_i\), on the original image regions $\Psi_i\bigl(L\times[-B,-1]\bigr)\subset M\times[-B,-1]$, satisfy
\[
    \sqrt{-1}\partial\bar\partial\widetilde u_i
    =
    \operatorname{Ric}_{\omega_i},
    \qquad
    (\partial_s-\Delta_{g_i(s)})\widetilde u_i=0.
\]
Applying local parabolic regularity on these image regions and then
pulling the resulting estimates back by the spacetime convergence
maps, a diagonal argument gives $v_i(\cdot,-1)
    =
    \Psi_{i,-1}^*\widetilde u_i(\cdot,-1)
    \longrightarrow u$
in $C^\infty_{\mathrm{loc}}(X\setminus Z)$. Passing to the limit in the pulled-back complex Hessian equation gives $\sqrt{-1}\partial\bar \partial u
    =
    \operatorname{Ric}_{\omega}$ on $X\setminus Z$. Let \(f\) be the shrinker potential. Then $\phi:=u+f$ satisfies
\begin{equation}\label{ppp}
    \sqrt{-1}\partial\bar\partial\phi
    =
    \omega
    \qquad\text{on }X\setminus Z.
\end{equation}

We next extend \(\phi\) across \(Z\). Choose a connected relatively
compact domain \(\Omega\Subset X\) with smooth boundary such that $Z\Subset\Omega$ and $\partial\Omega\cap Z=\varnothing$. Since \(E_i\to Z\) in the Hausdorff topology, $E_i\Subset\Omega$ for all sufficiently large \(i\). The convergence maps are defined on a neighborhood of \(\overline\Omega\) for all large \(i\). Define $\phi_i:=v_i(\cdot,-1)+f$ on $\Omega$. By Proposition \ref{heat}
\[
 \sqrt{-1}\partial_{J_i}\bar\partial_{J_i}\phi_i
 =
 \beta_i+2\pi(n-1)[E_i],
\]
where $\beta_i
        :=
        \operatorname{Ric}_{\bar\omega_i(-1)}
        +
        \sqrt{-1}
        \partial_{J_i}\bar\partial_{J_i}f$. We note that $\phi_i$ has logarithmic singularities along $E_i$ and $\phi_i \in L^1_{loc}$. The smooth convergence implies $\beta_i
        \longrightarrow\omega$ smoothly on $\overline\Omega$, consequently, $\beta_i>0$ on \(\overline\Omega\) for all sufficiently large \(i\). Since
\([E_i]\) is a positive \(J_i\)-holomorphic current, \(\phi_i\) is
\(J_i\)-plurisubharmonic on a neighborhood of
\(\overline\Omega\).

Since \(\partial\Omega\Subset X\setminus Z\), the locally smooth
convergence $\phi_i\longrightarrow\phi$ on $X\setminus Z$ gives $\sup\nolimits_{\partial\Omega}\phi_i\leq C$ for sufficiently large \(i\). The maximum principle for
plurisubharmonic functions therefore yields
\[  
        \sup\nolimits_{\Omega}\phi_i
        \leq
        \sup\nolimits_{\partial\Omega}\phi_i
        \leq C.
\]
Passing to the limit on \(\Omega\setminus Z\), we obtain $\sup\nolimits_{\Omega\setminus Z}\phi\leq C$. Thus \(\phi\) is locally bounded above near \(Z\). The extension theorem of plurisubharmonic
functions \cite[Theorem~5.24]{Dema} shows that there exists unique plurisubharmonic extension of \(\phi\) to \(X\).

Let \(A\subset X\) be an irreducible positive-dimensional compact
analytic subvariety. Suppose that \(A\not\subset Z\). Then $A_{\mathrm{reg}}\setminus Z\neq\varnothing$, and the restriction \(\phi|_A\) is not identically equal to \(-\infty\). Since \(A\) is compact and irreducible, the maximum
principle for plurisubharmonic functions implies that \(\phi|_A\) is constant. On the other hand, for any regular point of $A\setminus Z$, we have $\sqrt{-1}\partial\bar\partial
    \bigl(\phi|_{A}\bigr)>0$, contradicting the constancy of \(\phi|_A\). Therefore $A\subset Z$.
\end{proof}

\section{Algebraic part}
\subsection{Polarized Fano fibration and birational contraction}
By Section 2.2, we know any Kähler Ricci shrinker admits a polarized Fano fibration structure
\[
        q:X\longrightarrow W=\Spec R_X.
\]
Here \(q\) is a $T$-equivariant surjective projective morphism, \(q_*\OO_X=\OO_W\), \(W\) is a normal polarized
affine cone with a unique torus fixed point \(o\), and \(-K_X\) is
\(q\)-ample.

\begin{proposition}\label{sz}
The morphism \(q\) is birational and
\[
        q^{-1}(o)=Z,
        \qquad
        q:X\setminus Z\xrightarrow{\;\simeq\;}W\setminus\{o\}.
\]
\end{proposition}

\begin{proof}
Suppose that \(\dim W<n\). Choose \(x\in X\setminus Z\), and let
\(F_x\) be an irreducible component of the fiber \(q^{-1}(q(x))\)
containing \(x\).  Since \(q\) is projective, \(F_x\) is a positive-dimensional
compact analytic subvariety of \(X\). It meets \(X\setminus Z\), contrary
to Proposition~\ref{maximalcompact}. Thus $\dim W=n$.

Since $X$ is smooth and $W$ is a normal variety, and $q$ is projective and satisfies $q_{*}\mathcal{O}_X=\mathcal{O}_W$, which forces all fibers of $q$ to be connected \cite[Corollary 11.3]{hartshorne2013algebraic}. Since $\dim X=\dim W=n$, the general fiber of $q$ is a single point. By the standard relation between the degree of a morphism and the cardinality of its general fiber, the field extension degree $[\mathbb{C}(X):\mathbb{C}(W)]$ equals $1$, so $\mathbb{C}(X)\cong\mathbb{C}(W)$ and $q$ is birational.

By Proposition~\ref{maximalcompact}, \(Z\) is the maximal compact analytic subset of \(X\), therefore \(Z\) is invariant under biholomorphism of \(X\), in particular under the $T$-action. Moreover, \(q(Z)\) is a compact
analytic subset of the affine variety \(W\). Hence \(q(Z)\) is
finite. By Lemma~\ref{analy}, \(Z\) is connected, so \(q(Z)\) is
connected and therefore consists of a single point: $q(Z)=\{o'\}$.The morphism \(q\) is $T$-equivariant and \(Z\) is torus invariant, so \(o'\) is fixed by the torus. By the uniqueness of $T$-fixed point in \(W\), $o'=o$. Thus $Z\subset q^{-1}(o)$.

Due to every fiber of \(q\) is connected. Every
positive-dimensional irreducible component of \(q^{-1}(o)\) is
contained in \(Z\). If there exists a point $p\in q^{-1}(o)\setminus Z$, then the irreducible component of \(q^{-1}(o)\) containing \(p\)
would be zero-dimensional. Such a component is an isolated
connected component of fiber. Since
\(Z\subset q^{-1}(o)\) is nonempty, this would contradict the
connectedness of \(q^{-1}(o)\). Therefore $q^{-1}(o)=Z$.

For every \(w\in W\setminus\{o\}\), the fiber \(q^{-1}(w)\) contains
no positive-dimensional irreducible component: otherwise that
component would be contained in \(Z\), whereas \(q(Z)=\{o\}\).
Consequently, the restriction $q:X\setminus Z\longrightarrow W\setminus\{o\}$ is projective and quasi-finite, hence finite. It is also birational,
and \(W\setminus\{o\}\) is normal. A finite birational morphism onto a
normal variety is an isomorphism. Hence $q:X\setminus Z\xrightarrow{\;\simeq\;}W\setminus\{o\}$.
\end{proof}
Combined with the relative cone theorem \cite[Theorem~3.25]{kollar}, we can characterise the birational morphism $q$ more precisely. Since \(X\) is smooth and \(-K_X\) is \(q\)-ample, we select \(H=-K_X\), then 
\[
    \overline{\operatorname{NE}}(X/W)
    =
    \sum_{j=1}^{N}\mathbb R_{\geq0}[C_j],
\]
where the \(C_j\) are rational curves contracted by \(q\). And for every extremal ray $R\subset\overline{\operatorname{NE}}(X/W)$, there exist a normal variety \(X_R\) and projective morphisms
\[
        \varphi_R:X\longrightarrow X_R,
        \qquad
        \rho_R:X_R\longrightarrow W,
\]
such that $q=\rho_R\circ\varphi_R$. And for every irreducible curve $C\subset X$, $\varphi_R(C)=\mathrm{pt}$ iff $[C]\in R$.
\begin{lemma}
    The above \(\varphi_R\) is birational and $\operatorname{Exc}(\varphi_R)\subset Z$. In fact, \(\varphi_R\) is an isomorphism away from \(Z\).
\end{lemma}
\begin{proof}
    We set $W^\circ=W\setminus\{o\}$ and $X^\circ=X\setminus Z$, by Proposition \ref{sz}, $q|_{X^\circ}$ is an isomorphism between $X^\circ$ and $W^\circ$. Let $X_R^\circ:=\rho_R^{-1}(W^\circ)$, If \(y\in X_R^\circ\), choose \(x\in X\) with \(\varphi_R(x)=y\). Then $q(x)=\rho_R(y)\in W^\circ$, so \(x\in X^\circ\). It follows that $\varphi_R(X^\circ)=X_R^\circ$. Moreover, the morphism $(q|_{X^\circ})^{-1}\circ
        \rho_R|_{X_R^\circ}$ is inverse to \(\varphi_R|_{X^\circ}\). Hence $\varphi_R|_{X^\circ}:
        X^\circ\xrightarrow{\;\simeq\;}X_R^\circ$. Consequently, \(\varphi_R\) is birational and $\operatorname{Exc}(\varphi_R)\subset Z$.
\end{proof}

\subsection{Relative Mori cone rigidity}

We first estimate the lower bound of the length $\ell(R)$ of the relative extremal ray $R$, where 
\begin{align*}
   l(R):=\min\{-K_X\cdot C \mid C\subset X \text{ is a rational curve, } [C]\in R\}.
\end{align*}
\begin{lemma}
\label{lower}
For every relative extremal ray
\(R\subset \overline{\operatorname{NE}}(X/W)\), one has $l(R)\geq n-1$.
\end{lemma}

\begin{proof}
Let \(C\subset X\) be a rational curve whose class belongs to \(R\).
Since \(C\) is a positive-dimensional compact analytic subvariety,
Proposition~\ref{maximalcompact} gives \(C\subset Z\). Choose a relatively compact open neighborhood \(U\Subset X\) of
\(\lvert C\rvert\), and let $\Phi_i\colon U\longrightarrow M$ be the Cheeger--Gromov embeddings. Set $J_i:=\Phi_i^*J_M$ and $J_i\longrightarrow J_\infty$ in $C_{loc}^\infty(X)$.
Fix a sufficiently small holomorphic coordinate ball
\(\Omega\Subset Y\) centered at \(p\), and write $\widetilde{\Omega}:=\pi^{-1}(\Omega)\subset M$. After enlarging \(U\) slightly and applying Lemma \ref{loc}
to its compact closure, we obtain $\Phi_i(\lvert C\rvert)\subset\widetilde{\Omega}$ for all sufficiently large \(i\).

 We consider the push-forward of its integral
fundamental class $\alpha_i:=(\Phi_i)_*[C]$ in $H_2(\widetilde{\Omega};\mathbb Z)$. The coordinate ball \(\widetilde{\Omega}\) deformation
retracts onto the exceptional divisor.
Therefore $H_2(\widetilde{\Omega};\mathbb Z)\simeq \mathbb Z[\ell]$, where \(\ell\subset E\) is a projective line. Hence $\alpha_i=m_i[\ell]$
for some \(m_i\in\mathbb Z\). Since \(J_i\to J_\infty\) uniformly on \(U\), the almost complex
structures \(J_i\) and \(J_\infty\) are homotopic on \(U\) for 
sufficiently large \(i\). Consequently, $c_1(TU,J_i)=c_1(TU,J_\infty)$ in $H^2(U;\mathbb Z)$.

Moreover, by \(J_i=\Phi_i^*J_M\), we have the isomorphism of complex vector bundles $(TU,J_i)\simeq \Phi_i^*(TM,J_M)\big|_{\Phi_i(U)}$. The naturality of the first Chern class therefore gives
\begin{align*}
    -K_X\cdot C
    =\big\langle c_1(TX,J_\infty),[C]\big\rangle
    =\big\langle c_1(TU,J_i),[C]\big\rangle
    =\big\langle c_1(TM,J_M),(\Phi_i)_*[C]\big\rangle
    =m_i\big\langle c_1(TM,J_M),[\ell]\big\rangle .
\end{align*}

For the blow-up of a smooth point, we have $K_M=\pi^*K_Y+(n-1)E$. Since \(\pi(\ell)=p\) and \(E\cdot\ell=-1\), it follows that
$\big\langle c_1(TM),[\ell]\big\rangle
       =-K_M\cdot\ell
       =n-1$. Thus $-K_X\cdot C=(n-1)m_i$. In particular, \(m_i\) is independent of \(i\) for all sufficiently
large \(i\); denote this integer by \(m(C)\).

Finally, since \(-K_X\) is \(q\)-ample and \(C\) is contracted by
\(q\), we have $-K_X\cdot C>0$. As \(n\geq2\), the preceding identity implies \(m(C)>0\), hence
\(m(C)\geq1\) and $-K_X\cdot C\geq n-1$, therefore $l(R)\geq n-1$.
\end{proof}

\begin{proposition}
\label{prop:length}
For every relative extremal ray $R\subset\overline{\operatorname{NE}}(X/W)$, one has $l(R)=n-1$. Moreover, every irreducible component \(E\) of \(\operatorname{Exc}(\varphi_R)\) has dimension \(n-1\) and is
contracted by \(\varphi_R\) to a point.
\end{proposition}

\begin{proof}
Let \(E\) be an irreducible component of
\(\operatorname{Exc}(\varphi_R)\), and let \(F\subset E\) be an
irreducible component of a nontrivial fiber of \(\varphi_R\).
The relative Ionescu--Wi\'sniewski fiber-locus inequality
\cite[\S1.B, (1.1)]{HN} gives
\[
        \dim E+\dim F
        \geq n+l(R)-1.
\]
Since \(\varphi_R\) is birational, we know $\dim E\leq n-1$ and $\dim F\leq n-1$. On the other hand, Lemma~\ref{lower} gives $l(R)\geq n-1$. Consequently,
\[
        2n-2
        \geq \dim E+\dim F
        \geq n+l(R)-1
        \geq 2n-2.
\]
All inequalities are equalities. Hence $l(R)=n-1$ and $\dim E=\dim F=n-1$. Since \(F\subset E\) and both are irreducible of the same dimension,
we have \(F=E\). It follows that $\dim\varphi_R(E)=0$, so \(E\) is contracted to a point.
\end{proof}

\begin{remark}
\label{rem:relative-AO}
Then we use the standard quasi-projective form of the
Andreatta--Occhetta characterization. Namely, a projective elementary
divisorial contraction from a smooth quasi-projective variety whose
nontrivial fibers have dimension \(l(R)\) is locally the blow-up of a
smooth center of codimension \(l(R)+1\) \cite[\S1.B, the paragraph following
Theorem~1.4]{HN}, with reference to \cite[Theorem~5.1]{AO}.
Although the latter is stated for projective varieties, its proof
localizes over the base; in the present setting all relevant curves
are contained in the projective exceptional fiber.
\end{remark}
The preceding length computation and
Remark~\ref{rem:relative-AO} immediately yield the following.
\begin{proposition}
\label{prop:rigidity}
For every relative extremal ray $R\subset\overline{\operatorname{NE}}(X/W)$,
the associated contraction $\varphi_R:X\longrightarrow X_R$ has a unique compact exceptional prime divisor \(D_R\subset X\).
Moreover, in a neighborhood of \(D_R\), the morphism \(\varphi_R\)
is the ordinary blow-down of \(D_R\) to a smooth point. In particular,
\[
        D_R\simeq\PP^{n-1},
        \qquad
        N_{D_R/X}\simeq\OO_{\PP^{n-1}}(-1).
\]
\end{proposition}

\section{The Proof of Main Theorem}
In this section we prove the main theorem \ref{main} by combining the analytic and algebraic results established earlier.
\begin{proposition}
\label{prop:saturation}
The limiting analytic cycle is reduced and irreducible. More precisely,
\[
        \mathcal Z=[D_R], \qquad  Z=D_R\simeq\PP^{n-1}.
\]
\end{proposition}

\begin{proof}
By Proposition~\ref{prop:rigidity}, $D_R\simeq\PP^{n-1}$ and $N_{D_R/X}\simeq\OO_{\PP^{n-1}}(-1)$. Hence
Proposition~\ref{maximalcompact} gives $D_R\subset Z$. Since \(Z\) has pure dimension \(n-1\) and
\(\dim D_R=n-1\), the irreducible analytic set \(D_R\) is an
irreducible component of \(Z\). Thus,
\[
        \mathcal Z
        =m_R[D_R]+\sum_{\alpha\neq R}m_\alpha[Z_\alpha],
        \qquad
        m_R\geq1.
\]
The shrinker equation gives $[\omega_\infty]=c_1(X)$ with our normalization. Since \(D_R\) is the exceptional divisor of
the blow-up of a smooth point, we know $c_1(X)|_{D_R}=(n-1)h$, where $h=c_1\bigl(\OO_{\PP^{n-1}}(1)\bigr)$. Therefore
\[
\begin{aligned}
        \operatorname{Vol}_{\omega_\infty}(D_R)
        =
        \frac{1}{(n-1)!}
        \int_{D_R}\omega_\infty^{n-1}  
        =
        \frac{(n-1)^{n-1}}{(n-1)!}
        \int_{\PP^{n-1}}h^{n-1}
        =V_0.
\end{aligned}
\]
On the other hand, the Proposition \ref{cycle} gives $\mathbf M_{\omega_\infty}(\mathcal Z)=V_0$. Consequently,
\begin{align*}
        V_0
        =
        \mathbf M_{\omega_\infty}(\mathcal Z)
        =
        m_R\operatorname{Vol}_{\omega_\infty}(D_R)
        +\sum_{\alpha\neq R}
          m_\alpha\operatorname{Vol}_{\omega_\infty}(Z_\alpha)
        \geq m_RV_0,
\end{align*}
it follows that $m_R=1$. And every positive-dimensional compact analytic component has strictly
positive \(\omega_\infty\)-volume. Hence no additional component can
occur, and therefore $\mathcal Z=[D_R]$. Taking supports gives $Z=D_R\simeq\PP^{n-1}$.
\end{proof}

\begin{proposition}
\label{prop:q-is-blowdown}
The relative cone \(\overline{\operatorname{NE}}(X/W)\) consists of
the single ray \(R\), and \(o\in W\) is smooth and \(q\) is the ordinary blow-up
of \(W\) at \(o\).
\end{proposition}

\begin{proof}
Every curve contracted by \(q\) is contained in $q^{-1}(o)=Z=D_R\simeq\PP^{n-1}$. Since $\overline{\operatorname{NE}}(\PP^{n-1})=\mathbb R_{\geq0}[\ell]$, all \(q\)-contracted curves have numerical classes in the same relative ray. Hence $\overline{\operatorname{NE}}(X/W)=R$.

The morphisms \(q\) and \(\varphi_R\) therefore contract precisely the
same curves. Since \(\varphi_R\) has connected fibers, \(q\) factors as $q= \overline q \circ \varphi_R$ for a projective birational morphism $\overline q:X_R\longrightarrow W$. The morphism \(\overline q\) is an isomorphism away from
\(b_R=\varphi_R(D_R)\), and $\overline q^{-1}(o)=\{b_R\}$. Thus \(\overline q\) is quasi-finite. Since it is projective, it is
finite. Both \(X_R\) and \(W\) are normal, so a finite birational
morphism between them is an isomorphism. Therefore $X_R\simeq W$ and $q=\varphi_R$.

By Proposition~\ref{prop:rigidity}, \(X_R\) is smooth at \(b_R\) and
\(\varphi_R\) is the ordinary blow-up of \(b_R\). Hence \(W\) is
smooth at \(o\), and \(q\) is the ordinary blow-up of \(W\) at \(o\).
\end{proof}

\begin{proof}[Proof of main theorem]
By Lemma~\ref{lem:smooth-cone} and Proposition~\ref{prop:q-is-blowdown}, we know 
\begin{align*}
    X\simeq\operatorname{Bl}_{o}W\simeq\operatorname{Bl}_{o}{\mathbb C^n},   
\end{align*}
by \cite[Theorem E]{CDS}, we know $X$ is the FIK shrinker.
\end{proof}

\paragraph{Acknowledgements.} The authors are grateful to Professors Andreas Höring and Zhixian Zhu for many helpful discussions and valuable advice on the algebraic-geometric aspects of this work. Zhenlei Zhang is supported by the National Natural Science Foundation of China (Grant No. 12531001)
\paragraph{Use of AI.}
The main ideas, arguments, and overall direction of the paper were
developed by the authors. AI tools (Chat GPT and Deepseek) were used during the preparation of
the manuscript as technical assistants for checking arguments,
locating relevant references, and improving the exposition. In particular, AI tools suggested estimating the length of a relative extremal ray and using the obtained equality case for the algebraic‑geometric rigidity argument.
All mathematical statements, hypotheses, references, proofs, and
conclusions were independently verified by the authors, who take full
responsibility for the contents of the paper.

\bibliographystyle{alpha}
\bibliography{ref.bib}

\section*{Author Information}
\noindent\textbf{Tongxin Xu}$^*$\\
School of Mathematical Sciences, Capital Normal University\\
Email: 2250501032@cnu.edu.cn
\vspace{1em}

\noindent\textbf{Zhenlei Zhang}$^\dagger$\\
School of Mathematical Sciences, Capital Normal University\\
Email: zhleigo@aliyun.com\\
\end{document}